\documentclass{amsart}

\usepackage{sseq,pgffor,sidenotes}
\usepackage{amsmath,amsthm,amssymb,array,enumerate,parskip,graphicx,etoolbox,tabularx}
\usepackage[all]{xy}

\usepackage{xcolor} 
\usepackage{hyperref}
\definecolor{dark-red}{rgb}{0.5,0.15,0.15}
\definecolor{dark-blue}{rgb}{0.15,0.15,0.6}
\definecolor{dark-green}{rgb}{0.15,0.6,0.15}
\hypersetup{
    colorlinks, linkcolor=dark-red,
    citecolor=dark-blue, urlcolor=dark-green
}

\usepackage[nameinlink,capitalise,noabbrev]{cleveref}

\newcount\tmp
\newcommand{\Z}{\mathbb{Z}}
\newcommand{\Q}{\mathbb{Q}}
\newcommand{\G}{\mathbb{G}}
\newcommand{\R}{\mathbb{R}}
\newcommand{\C}{\mathbb{C}}
\newcommand{\Hom}{\mathop{\mathrm{Hom}}}
\newcommand{\Ext}{\mathrm{Ext}}
\newcommand{\sdet}{\langle \text{det} \rangle}
\newcommand{\Mapc}{\text{Map}^c}
\usepackage{comment}

\renewcommand{\theequation}{\thesection.\alph{equation}}

\DeclareMathOperator{\Spec}{Spec}
\DeclareMathOperator{\OX}{\mathcal{O}_{\mathcal{X}}}
\DeclareMathOperator{\IZ}{I_{\Z}}
\DeclareMathOperator{\Pic}{Pic}

\DeclareMathOperator{\Ker}{Ker}
\DeclareMathOperator{\holim}{holim}

\newtheorem{theorem}{Theorem}[section]

\newtheorem{lemma}[theorem]{Lemma} 
\newtheorem{cor}[theorem]{Corollary}
\newtheorem{example}[theorem]{Example}
\newtheorem{prop}[theorem]{Proposition} \theoremstyle{definition}
\newtheorem{rem}[theorem]{Remark} 
\newtheorem{definition}[theorem]{Definition}

\Crefname{cor}{Corollary}{Corollaries}
\Crefname{conjecture}{Conjecture}{Conjectures}
\Crefname{rem}{Remark}{Remarks}
\Crefname{prop}{Proposition}{Propositions}	
\Crefname{question}{Question}{Questions}

\title{$K$-theory, reality, and duality}
\author{Drew Heard}
\address{Melbourne University, Australia}
\email{d.heard@student.unimelb.edu.au}
\author{Vesna Stojanoska}
\address{Massachusetts Institute of Technology, Cambridge MA}
\email{vstojanoska@math.mit.edu}
\thanks{The second author is partially supported by NSF grant DMS-1307390}

\begin{document}

\begin{abstract}
We show that the real $K$-theory spectrum $KO$ is Anderson self-dual using the method previously employed in the second author's calculation of the Anderson dual of $Tmf$. Indeed the current work can be considered as a lower chromatic version of that calculation. Emphasis is given to an algebro-geometric interpretation of this result in spectrally derived algebraic geometry. We finish by applying the result to a calculation of 2-primary Gross-Hopkins duality at height 1, and obtain an independent calculation of the group of exotic elements of the $K(1)$-local Picard group.
\end{abstract}
\maketitle

\section{Introduction}

The purpose of this note is to explain some results on duality for complex and real $K$-theory and their ramifications for the $K(1)$-local Picard group at the prime 2. The results obtained herein are not new to experts, and the major ingredients can be found in the literature. However, to the best of the authors' knowledge, those ingredients have not been previously blended together in the manner presented in this paper. That perspective and approach was of great use as a guiding example for the second author in her work on duality for topological modular forms~\cite{stojanoska2012duality}.

There are two main forms of duality which algebraic topologists have used in $K(n)$-local settings, namely, Spanier-Whitehead duality, which is simply the monoidal duality, and Gross-Hopkins duality, which is a $K(n)$-local analogue of Pontryagin or Brown-Comenetz duality. We will be interested in an integral version of the latter, namely Anderson duality. This is not $K(n)$-local, though upon such localization is closely related to Gross-Hopkins duality. Instead, it is defined on the category of spectra and makes surprising appearances in geometry and the study of quantum field theories \cite{hopkins2005quadratic, FreedMooreSegal}.

Consider the $C_2$-action on the periodic $K$-theory spectrum $KU$ via complex conjugation. The main computational result in this paper is the following
\theoremstyle{plain}
\newtheorem*{thm:Andersondual}{\Cref{thm:Andersondual}}

\begin{thm:Andersondual}
The Anderson dual $I_\Z KU$ is $C_2$-equivariantly equivalent to $\Sigma^4KU$.
\end{thm:Andersondual}

 The real $K$-theory spectrum $KO$ is the spectrum of homotopy fixed points $KU^{hC_2}$, and the above duality is reflected in the following result as a ``wrong side computation."

\newtheorem*{thm:ushriek}{\Cref{thm:ushriek}}
\begin{thm:ushriek}
The forgetful map \[u_*:(KO\text{-mod})\to (S\text{-mod})\] has a right adjoint $u^!=F(KO,-)$ such that for a spectrum $A$
 \[ u^!A =F(KO, A) \simeq F ( I_\Z A, \Sigma^4 KO).\]
\end{thm:ushriek}

This theorem has a derived algebro-geometric interpretation which we pursue in~\Cref{sec:andersonduality}.

The first step of the proof of~\Cref{thm:Andersondual} consists of a simple calculation that $KU$ is Anderson self-dual. We then proceed to develop a descent strategy for deducing the above-stated result. The fact that $KO$ is the homotopy fixed points of $KU$ under the conjugation action is insufficient for descending duality. Notwithstanding, we show that $KO$ is also equivalent to the homotopy orbits of $KU$ under the same $C_2$ action by proving that the associated Tate spectrum vanishes. This allows a calculation of the Anderson dual of $KO$, and from there we complete the proof of~\Cref{thm:Andersondual}.

As an application of the above theorems, we use the relationship between Anderson duality and $K(1)$-local Gross-Hopkins duality to independently conclude the well-known fact that the $K(1)$-local category contains ``exotic" elements in its Picard group.

The organization of this paper is as follows. In~\Cref{sec:andersonbasics} we define Anderson duality, and in~\Cref{sec:andersonduality} we give a detailed algebro-geometric interpretation of the duality in general, as well as an  interpretation of~\Cref{thm:ushriek} using derived algebraic geometry. Note that the proof of~\Cref{thm:ushriek}, while included in~\Cref{sec:andersonduality}, depends on the results of Sections \ref{sec:norm} through \ref{sec:KOanderson}; these Sections build to the proof of~\Cref{thm:Andersondual}. In particular, in~\Cref{sec:HFPSS} we use equivariant homotopy theory to deduce, from scratch, the differentials in the homotopy fixed point spectral sequence for $KO\simeq KU^{hC_2}$. Finally, in~\Cref{sec:picard} we study the implications of~\Cref{thm:Andersondual} for the invertible $K(1)$-local spectra and deduce that an exotic such spectrum must exist.

\begin{rem}
Where possible we have included information that, though known to experts, is perhaps not easily or at all found in the literature, such as the use of equivariant homotopy to determine a differential in a homotopy fixed point spectral sequence which we learned from Mike Hopkins and Mike Hill.
\end{rem}

\section*{Conventions}
We will denote homotopy fixed point and orbit spectral sequences by $\operatorname{HFPSS}$ and $\operatorname{HOSS}$ respectively. In all spectral sequence diagrams a box represents a copy of $\Z$ whilst a dot is a copy of $\Z/2$. Vertical lines represent multiplication by 2, whilst lines of slope 1 represent multiplication by $\eta$. All spectral sequences will be Adams indexed; that is, a class in cohomological degree $s$ and internal degree $t$ will be drawn in position $(t-s,s)$. 
By a commutative ring spectrum we will always mean a highly structured commutative ring.

\section*{Acknowledgements}
The authors thank Paul Goerss for his invaluable suggestions to greatly improve earlier versions of this work.
The second author also thanks Jeremiah Heller for helpful discussions about $C_2$-equivariant homotopy theory and for a careful proofreading of a draft of this document. 

\section{Background on Anderson duality}\label{sec:andersonbasics}

The functor on spectra $X \mapsto \Hom(\pi_{-*}X,\Q/\Z)$ is a cohomology theory since $\Q/\Z$ is injective. Let $I_{\Q/\Z}$ denote the representing spectrum. The Brown-Comenetz dual of $X$ is then defined as the function spectrum $I_{\Q/\Z}X = F(X,I_{\Q/\Z})$. In a similar way we can define $I_\Q$ to be the spectrum representing the functor $X \mapsto \Hom(\pi_{-\ast} X,\Q)$, and there is a natural map $I_{\Q} \to I_{\Q/\Z}$. We denote the homotopy fiber of this map by $I_\Z$. Of course, $I_\Q$ is the rational Eilenberg-MacLane spectrum. For a spectrum $X$ we then define $I_\Z  X$ as the function spectrum $F(X,I_\Z)$. This contravariant functor $\IZ$ on the category of spectra is known as Anderson duality, first used by Anderson~\cite{anderson1969universal}; see also~\cite{yosimura1975universal,hopkins2005quadratic}. We shall see that the representing spectrum $\IZ$ is the dualizing object in the category of spectra (see \Cref{sec:andersonduality}).

\begin{example}
If $A$ is a finite group, then the Brown-Comenetz dual of the Eilenberg-MacLane spectrum $HA$ is again $HA$, via a non-canonical isomorphism between $A$ and its Pontryagin dual. Since there are no non-zero maps $A\to \Q$, it follows that the Anderson dual $I_{\Z} HA $ is equivalent to $\Sigma^{-1} HA$.
\end{example}

For a spectrum $X$, we can use a short spectral sequence to calculate the homotopy of $I_\Z X$. By considering the long exact sequence in homotopy of the fiber sequence
\begin{equation}\label{eq:fibSeq}
I_\Z X \to I_\Q X \to I_{\Q/\Z} X,  
\end{equation}
we can form an exact couple and hence a spectral sequence, 
\begin{equation}\label{eq:andersonSS}
\text{Ext}^s_\Z(\pi_tX,\Z) \Rightarrow \pi_{-t-s}I_\Z X,
\end{equation}
where the spectral sequence is only non-zero for $s=0,1$. 

Often this leads to very simple calculations, for example in the case of $KU$. Here, of course, $\pi_* KU $ is the ring of Laurent polynomials $ \Z[u^{\pm 1}]$ on the Bott element $u$ of degree $2$. Hence we easily conclude that the homotopy groups $\pi_* I_\Z KU$ are $\Z$ in even degrees and $0$ in odd degrees. We can do more by observing the additional available structure; note that if $R$ is a ring spectrum, then $I_\Z R=F(R,I_\Z)$ is a module over $R$. The following easy observation has crucial applications.

\begin{prop}\label{prop:moduleeq}
Let $R$ be a homotopy ring spectrum\footnote{i.e. a ring object in the homotopy category} and let $M$ be an $R$-module such that $\pi_* M$ is, as a graded module, free of rank one over $\pi_* R$, generated by an element in $\pi_t M$. Then $M\simeq \Sigma^t R$.
\end{prop}

\begin{proof}
The generating element is represented a map of spectra $S^t \to M$ which, by the standard adjunction can be extended to a map of $R$-modules $S^t\wedge R \to M$ using the $R$-module structure of $M$.
The assumptions ensure that $\varphi$ is an equivalence.
\end{proof}

Since $ \Hom_\Z(\Z[u^{\pm 1},\Z])$ is a free $\Z[u^{\pm 1}]$-module of rank one, an immediate corollary of this proposition is that $I_\Z KU$ is equivalent to $KU$, i.e. $KU$ is Anderson self-dual.

Similarly, we can run the spectral sequence~\eqref{eq:andersonSS} for $I_\Z KO $ and conclude that $\pi_k I_\Z KO$ is $\Z$ when $k$ is divisible by $4$, $\Z/2$ when $k$ is congruent to $5$ and $6$ modulo $8$, and zero otherwise. This indicates that perhaps $I_\Z KO$ is a fourfold suspension of $KO$, which we show is true in~\Cref{theorem:andersonKO}. However, determining the $KO$-module structure is tricky from the direct approach that works for $I_\Z KU$ because multiplication by $\eta$ on the dual non-torsion classes cannot be detected algebraically. Thus we proceed using a more sophisticated approach.\footnote{The reader interested in the direct approach is referred to the appendix of \cite{FreedMooreSegal} for details.}

A natural question that arises at this point is whether we can use the self-duality of $KU$ to deduce any information about the Anderson dual of the real $K$-theory spectrum $KO$, using (only) that $KO$ is the homotopy fixed point spectrum $KU^{hC_2}$, where $C_2 = \langle c | c^2=1 \rangle$ acts on $KU$ by complex conjugation.
In other words, we would like to know 
\begin{equation}\label{eq:AD}
I_\Z KO \simeq I_\Z KU^{hC_2} \simeq F(KU^{hC_2},I_\Z)
\end{equation}
as a $KO$-module. Note that homotopy fixed points are the right, not left adjoint to the forgetful functor, hence we cannot deduce much solely from the definitions and the self-duality of $KU$. However, we show in~\Cref{sec:tate} that $KO$ is also the homotopy orbit spectrum $KU_{hC_2}$, and then proceed to the identification of $I_\Z KO$.

\section{Algebro-geometric meaning of Anderson duality}\label{sec:andersonduality}

Before proceeding to the computational specifics for $K$-theory, we would like to put the above duality discussion in a broader perspective. Namely, we claim that Anderson duality naturally occurs from a categorical viewpoint, as it is related to the existence of ``wrong side" adjoints, or from an algebro-geometric viewpoint as such adjoints signify the compactness of certain geometric objects.\footnote{For example, compact manifolds have Poincar\'e duality, and proper smooth schemes have Serre duality.}

Perhaps the first notion of duality one encounters is that for vector spaces over a field $k$; the dual $V^\vee$ of a vector space $V$ is the space of linear maps from $V$ to the base field $k$, and has the universal property that linear maps $W\to V^\vee$ are in bijection with bilinear pairings $W\otimes V\to k$. If we restrict our attention to finite dimensional vector spaces, we can recover $V$ from its dual as $V\cong (V^\vee)^\vee$.

As is well known, if we try to directly imitate this situation in the more general case of (finitely generated) modules over a ring $R$ and define a ``naive" dual of $M$ as $\Hom_R(M,R)$, perhaps the first problem we encounter is the inexactness of the functor $\Hom$. For example, if $R=\Z$ and $M$ is finite, $\Hom_\Z(M,\Z)$ is zero, so the naive dual cannot recover $M$. The initial obstacle is easily resolved by passing to the derived category of (complexes) of $R$-modules. One observes that for this to work properly, $R$ ought to have finite injective dimension as a module over itself.	 If it does not, we may still be able to define good dual modules by mapping into some other object instead of $R$. Hence we arrive at the following definition.

\begin{definition}\label{def:dualizingalg}
An $R$-module\footnote{i.e. an object of the derived category of complexes of $R$-modules} $D$ is dualizing if 
\begin{enumerate}
\item\label{c1} for any finitely generated module $M$, the natural double duality map \[M\to \Ext_R(\Ext_R(M,D),D)\] is an isomorphism, and
\item\label{c2} $D$ has finite injective dimension; in other words, if a module $M$ is bounded below, then its $D$-dual $\Ext_R(M,D)$ is bounded above.
\end{enumerate}
Then $I_D M:=\Ext_R(M,D)$ is called the dual of $M$, and the contravariant functor $I_D$ represented by $D$ is called a (representable) duality functor.
\end{definition}
Note that if condition \eqref{c2} holds, and the map $R\to \Ext_R(D,D)$ is an isomorphism, i.e. if condition \eqref{c1} holds for $M=R$, then \eqref{c1} holds for any $M$; thus checking whether an object $D$ is dualizing requires relatively little work.

\begin{example}\label{exam:dualz} The following example may provide some motivation for the definition of the Anderson dual as the fiber of $I_\Q \to I_{\Q/\Z}$. The complex $\Q \to \Q/\Z$ is an injective resolution of $\Z$; hence we can use it to compute $I_\Z(M)=\Ext_\Z(M,\Z)$ for any abelian group $M$ as the two-term complex $I_\Q(M)=\Hom_\Z(M,\Q)\xrightarrow{\pi} \Hom_\Z(M,\Q/\Z)=I_{\Q/\Z}(M)$. If $M$ is a finite abelian group, then $I_\Z(M)$ and $ I_{\Q/\Z}(M)$ are equal, up to a shift, so the double duality map $M\to I_\Z(I_\Z(M))$ is an isomorphism. For finitely generated non-torsion $M$, the map $\pi$ is surjective, so $I_\Z(M)=\Hom_{\Z}(M,\Z)$ and the double duality map $M\to I_\Z(I_\Z(M))$ is again an isomorphism. Hence $\Z$ is a dualizing $\Z$-module.
\end{example}

A dualizing $R$-module in the non-derived setting need not exist; for example the category of $\Z$-modules (i.e.~abelian groups) does not have a dualizing module. Nevertheless, its derived category does has a dualizing module: $\Z$ itself serves that role (as in~\Cref{exam:dualz}). This is a significant example of duality in algebra as $\Z$ is an initial object in rings, and dualities for (derived) categories of $R$-modules can be obtained by studying the map $u:\Z\to R$ (or rather, the induced map on the corresponding categories.)

To be more precise, consider the forgetful functor $u_*$ between the derived categories of $R$-modules and $\Z$-modules. It always has a left adjoint $u^*$ given by tensoring up with $R$. However, suppose $u_*$ also has a right adjoint $u^!$. Then for any abelian group $A$ and $R$-module $M$, there is an isomorphism
\[\Ext_\Z(u_*M,A)\cong \Ext_R(M,u^!A); \]
which immediately gives that $u^!\Z$ is a good candidate for a dualizing $R$-module. Indeed, we have that 
\[I_{u^!\Z} (I_{u^! \Z} M) = \Ext_R( \Ext_R(M, u^!\Z ), u^!\Z ) \cong \Ext_\Z( \Ext_\Z(M, \Z ),\Z ) \cong M.\]

In the above discussion, there was no need to restrict ourselves to modules over commutative rings. Indeed, considering quasi-coherent modules over the structure sheaf of a scheme gave rise to Grothendieck-Serre duality; in the algebro-geometric world, a dualizing module is defined precisely as in the algebraic Definition \ref{def:dualizingalg}. Moreover, given a map of schemes $f:X\to Y$, such that $Y$ has a dualizing module $D_Y$ and the pushforward functor $f_*$ is faithful and has a right adjoint $f^!$, we get that $f^!D_Y$ is a dualizing module over $X$. For details, the reader is referred to \cite{hartshorne,neeman,fauskhumay}.

Dualizing objects can also be constructed in the category of ring spectra. Here the sphere is an initial object, and the category of $S$-modules is the category of spectra. (Note that this is already derived.) The following definition is due to Lurie~\cite{LurieDAGXIV}.
\begin{definition}
Let $A$ be a connective commutative ring spectrum, and let $K$ be an $A$-module (for example, as in~\cite{kriz2007rings}). Then $K$ is a dualizing $A$-module if
\begin{enumerate}
\item\label{item1} Each homotopy group $\pi_n K$ is a finitely generated module over $\pi_0 A$ and $\pi_n K$ vanishes for $n \gg 0.$ 
\item\label{item2}[Dualizing property] The canonical map $A \to F_A(K,K)$ is an equivalence.
\item\label{item3} The module $K$ has finite injective dimension: there is an integer $n$ such that if $M$ is an $A$-module with $\pi_i M = 0$ for $i > n$, then $\pi_iF_A(M,K) = 0$ for $i < 0$.
\end{enumerate}
\end{definition}
\begin{example}
We immediately see that the sphere is not a dualizing module over itself (in particular it does not satisfy the vanishing condition), but the Anderson spectrum $I_\Z$ plays this role. The proof is easy and can be found in~\cite[Example 4.3.9]{LurieDAGXIV}. Indeed, \eqref{item1} and \eqref{item3} are obvious from the definition, and \eqref{item2} follows from the fact that $\Z$ is the dualizing object in the derived category of $\Z$-modules. More precisely, duality in $\Z$-modules tells us that the homotopy groups $\pi_*F(I_\Z,I_\Z)\cong \Ext_\Z(\Ext_\Z(\pi_*S,\Z ))$ are isomorphic to $\pi_*S$ as a $\pi_*S$-module.
Similarly, the category of $p$-local spectra, i.e. modules over the $p$-local sphere spectrum, has a dualizing object which is given by the spectrum $I_{\Z_{( p)}}$, the homotopy fiber of the natural map $I_{\Q} \to I_{\Q/\Z_{( p)}}$.
\end{example}

\begin{rem}\label{rem:duals}
The above definition is suitable only for modules over \emph{connective} ring spectra, and therefore cannot be applied as is to $KU$ or $KO$-modules. However, if $R$ is \emph{any} ring spectrum, we can study dualizing $R$-modules relative to the unit map $u: S\to R$. Namely, the forgetful functor $u_*$ from $R$-modules to spectra has both a left and a right adjoint, the right being given by taking the function spectrum $F(R,-)$. Thus by the above formal reasoning, $u^! I_\Z = F(R,I_\Z)$ will have the dualizing property for $R$-modules.	
\end{rem}

\subsection{Homotopical duality for stacks}

We would like to patch together the geometric and homotopical notions of duality, thus we need a notion of a derived scheme which is a good notion of a locally ringed-by-spectra space. In fact, we shall not restrict ourselves to schemes, but consider (very simple) stacks, as we would like to have an object $BG$ for a finite group $G$. For our purposes, the following definition (which is in fact a theorem of Lurie's) will suffice.
\begin{definition}
A derived stack is an ordinary stack $(\mathcal{X},\mathcal{O}_0)$ equipped with a sheaf of commutative ring spectra $\OX$ on its small \'etale site such that
\begin{enumerate}
\item The pair $(\mathcal{X},\pi_0 \OX)$ is equivalent to $(\mathcal{X},\mathcal{O}_0)$, and
\item $\pi_k \OX$ is a quasi-coherent sheaf of $\pi_0 \OX$-modules.
\end{enumerate}
Here $\pi_k\OX$ is the sheafification of the presheaf $U\mapsto \pi_k(\OX(U))$. In this case say that $(\mathcal{X},\OX)$ or, abusively, $\mathcal{X}$, is a derived stack.
\end{definition}

\begin{example}
A commutative ring spectrum $R$ defines a derived stack $\Spec R$ whose underlying ordinary ``stack" is the affine scheme $\Spec \pi_0 R$. In particular, the sphere spectrum $\Spec S$ derives $\Spec \Z$, as does the Eilenberg-MacLane spectrum $H\Z$. In fact, $\Spec H\Z$ is the usual $\Spec \Z$, but it is no longer terminal; now $\Spec S$ is the terminal object.
\end{example}

By an $\OX$-module we will mean a quasi-coherent sheaf $\mathcal{F}$ of $\OX$-module spectra, i.e. a sheaf of $\OX$-modules such that for any diagram
\[ \xymatrix{
U \ar[rr]\ar[rd] &&V\ar[ld]\\
&\mathcal{X},
}\]
in which $U$ and $V$ are derived affine schemes \'etale over $\mathcal{X}$, we get an equivalence
\[ \mathcal{F}(V) \underset{\OX(V)}{\wedge} \OX(U) \xrightarrow{\sim} \mathcal{F}(U). \]

We will consider the following situation. Let $(\mathcal{X},\OX)$ be a derived stack over the sphere spectrum $S$ and let 
\[f:\mathcal{X}\to \Spec S\]
be its structure map. 
The push-forward or global sections functor $f_*: (\OX\text{-mod}) \to (S\text{-mod})$ always has a left adjoint, $f^*$, the constant sheaf functor. However, in some situations, it also has a right adjoint, usually denoted $f^!$. When such $f^!$ exists, $f^!\IZ$ is a perfect candidate for a dualizing module for $\OX$-modules.\footnote{Compare~\Cref{rem:duals}}

\begin{example}
Let $R$ be a commutative ring spectrum, and let $u:\Spec R\to \Spec S$ be the structure map. Then $u_*$, which is the forgetful functor from $R$-modules to spectra has both a left adjoint $u^* = - \wedge R$ and a right adjoint $u^! = F(R, -)$. Hence $u^! I_\Z = F(R,I_\Z) $ has the dualizing property for $R$-modules.
\end{example}

\begin{example}
Let $R$ be a commutative ring spectrum with an action by a finite group $G$. Then $(BG,R)$ is an example of a derived stack, whereby modules over the structure sheaf are $R$-modules with a compatible $G$-action. If $M$ is such a module, then $f_* M = M^{hG}$ has a left adjoint given by smashing with $R$. However, $f_*$ may or may not have a right adjoint for non-trivial $G$. The existence of such is related to the Tate spectra that we recall in the next section. 
Note that in a $K(n)$-local setting the right adjoint of such an $f_*$ always exists by a theorem of Greenlees-Sadofsky \cite{GreenleesSadofsky} recently generalized by Hopkins-Lurie \cite{HopkinsLurie}.
\end{example}

\subsection{Duality for $K$-theory, geometrically}

We now describe how $K$-theory fits into this perspective; to begin with, we recall the setup from \cite[Appendix A]{lawson2012strictly}. The reader is referred to loc.cit. for more details and proofs.

Let $\mathcal{M}_{\G_m}/\operatorname{Spec}{\Z}$ denote the stack classifying group schemes which become isomorphic to $\G_m$ after a faithfully flat extension. These are also called forms of the multiplicative group. Since the automorphism group of $\G_m$ over $\Z$ is $C_2$, we get that $\mathcal{M}_{\G_m}$ is equivalent to the stack $BC_2$ classifying principal $C_2$-torsors; explicit details can be found in loc.cit.

There exists a sheaf $\mathcal{O}^\text{top}$ of commutative ring spectra on the stack $\mathcal{M}_{\G_m}$ such that to each \'etale map $\Spec R \to \mathcal{M}_{\G_m}$, the sheaf $\mathcal{O}^\text{top}$ assigns a complex orientable weakly even periodic commutative ring spectrum $\mathcal{O}^\text{top}( R)$ such that 
\begin{enumerate}[(a)]
\item $\pi_0 \mathcal{O}^\text{top} ( R)$=R, and
\item the formal completion of the form of $\G_m$ classified by $\Spec R \to \mathcal{M}_{\G_m}$ is isomorphic to the formal group law of $\mathcal{O}^\text{top}( R)$.
\end{enumerate}
 Lawson-Nauman show that the spectrally ringed stack $(\mathcal{M}_{\G_m}, \mathcal{O}^{\text{top}})$ (over the sphere spectrum $\Spec S$) is equivalent to $(BC_2,KU)$, where $C_2$ acts on $KU$ (by commutative ring maps) via complex conjugation. In particular, the category of $\mathcal{O}^{\text{top}} $-modules is equivalent to the category of $C_2$-equivariant $KU$ modules.

More specifically, let $\Spec R \to \mathcal{M}_{\G_m}$ be an \'etale map; and let $\tilde R$ denote the ring which fits in a pullback square
\[\xymatrix{ 
\Spec \tilde R \ar[r] \ar[d] & \Spec R \ar[d]\\
\Spec \Z \ar[r] & \mathcal{M}_{\G_m} \cong BC_2,
} \]
where the bottom map is the usual covering map of $BC_2$ (which, in particular, is an affine map). The ring $\tilde R $ constructed in this way has a free $C_2$-action and as such has a homotopically unique $C_2$-equivariant realization $S(\tilde R) $ according to \cite{BakerRichter}, with
\[\pi_* S(\tilde R) \cong \pi_* S\otimes_{\Z} \tilde R.\]
The value that the sheaf $\mathcal{O}^{\text{top}}$ takes on the affine $\Spec R \to \mathcal{M}_{\G_m}$ is then given by $(KU\wedge S(\tilde R))^{hC_2}$, where $C_2$ acts diagonally on the smash product. We have 
\[\pi_* \mathcal{O}^{top}( R) = H^0(C_2, \pi_*KU \otimes \tilde R) \]

For example, if $R=\Z$ and the map is the covering $\Spec \Z \to \mathcal{M}_{\G_m}$ (i.e. the map classifying $\G_m$ itself), $\tilde R \cong \Z[C_2] $, and $\mathcal{O}^{top}(\Z) = (KU[C_2])^{hC_2} \simeq KU$.

 Thus the global sections of $\mathcal{O}^{\text{top}}$ over $\mathcal{M}_{\G_m}$ are the homotopy fixed points $KU^{hC_2}$, i.e. the real $K$-theory spectrum $KO$. In fact, using Galois descent or the general affineness machinery of Mathew-Meier \cite[Ex.~6.1]{MeierMatthew}, one shows that the category of $\mathcal{O}^{\text{top}}$-modules, which we saw is equivalent to $C_2\text{-}KU$-modules, is further equivalent to the category of $KO$ modules as, by loc.cit., the derived stack $(BC_2,KU)$ is equivalent to $\Spec KO$. Note that in the category of $C_2\text{-}KU$-modules, the internal $\Hom$ is given by the function spectrum of non-equivariant $KU$-module maps equipped with the $C_2$-action by conjugation; we will denote this function spectrum by $F_{C_2\text{-}KU}$.

 \begin{theorem}\label{thm:ushriek}
 The following three equivalent statements are true.
 \begin{enumerate}[(i)]
 \item Let $u$ be the structure map $ \Spec KO \to \Spec S$. The push forward (i.e. forgetful) map 
 \[u_*:(KO\text{-mod})\to (S\text{-mod})\] has a right adjoint $u^!=F(KO,-)$ such that for a spectrum $A$
 \[ u^!A =F(KO, A) \simeq F ( I_\Z A, \Sigma^4 KO).\]
 In particular, $u^! I_\Z \simeq \Sigma^4 KO $ is the relative dualizing $KO$-module.
 
\item\label{part:thm:g} The push forward $g_*$ along the structure map $g:(BC_2,KU)\to \Spec S$ has a right adjoint $g^!$ given by
 \[ g^! A = F_{C_2\text{-}KU} (g^* (I_\Z A), \Sigma^4 KU ). \]
 
\item The push forward $f_*$ along $(\mathcal{M}_{\G_m},\mathcal{O}^{\text{top}}) \to \Spec S$ has a right adjoint $f^!$ such that
\[ f^! A = \Hom{ }_{\mathcal{O}^{\text{top}}} (f^*(I_\Z A),\Sigma^4 \mathcal{O}^{\text{top}} ).\]
 In particular, $f^! I_\Z \simeq \Sigma^4\mathcal{O}^{\text{top}} $ is the relative dualizing $\mathcal{O}^{\text{top}}$-module.
 \end{enumerate}
\end{theorem}
\begin{rem}
The proof of this theorem depends on~\cref{thm:Andersondual}, which is the main result of this paper and to whose proof~\cref{sec:norm} through~\cref{sec:KOanderson} are devoted. Thus the logical placement of the proof of this theorem is at the end of~\cref{sec:KOanderson}, but we have decided to include it here for better readability. The reader is advised to come back to the proof after consulting the statements of~\cref{thm:Andersondual} and~\cref{cor:FixedHtpyFixed}.
\end{rem}

\begin{proof}
That the statements are equivalent is immediate from the above discussion. We will 
prove~\cref{part:thm:g}. What we need to show is that for a spectrum $A$ and a $C_2\text{-}KU$ module M, there is a natural equivalence
\begin{align}\label{eq:adjunction}
 g_* F_{C_2\text{-}KU}(M, F_{C_2\text{-}KU} (g^* (I_\Z A), \Sigma^4 KU )) \simeq F(g_* M, A). 
 \end{align}
We split the proof in two parts. 
\begin{enumerate}[(a)]
\item When $A$ is the Anderson spectrum $I_\Z$, \eqref{eq:adjunction} reduces to showing that 
\[ F_{KU}(M,\Sigma^4 KU)^{hC_2} \simeq F(M^{hC_2}, I_\Z ).\]
However, $\Sigma^4 KU$ is equivariantly equivalent to $I_\Z KU$ by~\Cref{thm:Andersondual}, and $M^{hC_2}$ is equivalent to $M_{hC_2}$  by~\Cref{cor:FixedHtpyFixed}. The latter gives an isomorphism $F(M^{hC_2},I_\Z) \simeq F(M_{hC_2},I_\Z) \simeq F(M,I_\Z)^{hC_2}$  which in turn implies that we have a chain of equivalences
\[ F_{KU}(M,\Sigma^4 KU)^{hC_2} \simeq F_{KU}(M,F(KU,I_\Z) )^{hC_2} \simeq F(M,I_\Z)^{hC_2} \simeq F(M^{hC_2}, I_\Z ).\]
\item For a general $A$, we note that the map $A \to F( F(A,I_\Z),I_\Z)$ is an equivalence, so we substitute $A$ by its double dual. For a $C_2\text{-}KU$ module $M$, since $M_{hC_2}\simeq M^{hC_2} $ (by~\Cref{cor:FixedHtpyFixed}) we have that
\[ M^{hC_2} \wedge F(A,I_\Z) \simeq (M\wedge F(A,I_\Z))^{hC_2},\]
implying
\begin{align*}
 F(M^{hC_2}, A) & \simeq   F\big(M^{hC_2}, F(F(A,I_\Z),I_\Z) \big)  \simeq F(M^{hC_2}\wedge F(A,I_\Z),I_\Z)\\
 & \simeq F\big( (M\wedge F(A,I_\Z))^{hC_2},I_\Z \big). 
 \end{align*}
By part (a), $ F\big( (M\wedge F(A,I_\Z))^{hC_2},I_\Z \big) $ is equivalent to 
\[ F_{KU} \big( M\wedge F(A,I_\Z), \Sigma^4 KU \big)^{hC_2}, \]
proving our result.
\end{enumerate}
\end{proof}

\section{The norm cofibration sequence}\label{sec:norm}
Suppose a finite group $G$ acts on a module $M$. The map on $M$ given by \[m \mapsto \sum_{g\in G} gm\]
factors through the orbits $M_G$ and lands in the invariants $M^G$, giving the norm $N_G:M_G \to M^G$.
The Tate cohomology groups are then defined as
\begin{equation}\label{eq:tateDef}
\hat{H}^n(G;M) = 
\begin{cases}
	H^n(G;M) &\text{ for } n \ge 1 \\
	\operatorname{coker}(N_G) &\text{ for } n = 0 \\
	\text{ker} (N_G) &\text{ for } n = -1 \\
	H_{-n-1}(G;M) &\text{ for } n \le -2.
\end{cases}
\end{equation}
	
In the same vein, Greenlees and May~\cite{greenlees1995generalized} have assigned a Tate spectrum to a $G$-spectrum $X$. The story starts with the cofiber sequence
\[ EG_+ \to S^0 \to \tilde EG, \]
which gives rise to the following commutative diagram of equivariant maps 
\[\xymatrix{ EG_+ \wedge X \ar[r] \ar[d] & X \ar[r] \ar[d] & \tilde EG \wedge X \ar[d] \\
EG_+ \wedge F(EG_+,X) \ar[r] & F(EG_+,X) \ar[r] & \tilde EG \wedge F(EG_+,X),
} \]
in which the rows are cofiber sequences.
 The left vertical arrow is always a $G$-equivalence. Upon taking $G$-fixed points and using the Adams isomorphism \cite[XVI.5.4]{may1996equivariant}, the diagram becomes 
\begin{align}\label{eq:NormCofibration}
\xymatrix{ X_{hG} \ar[r] \ar@{=}[d] & X^G \ar[r] \ar[d] & \Phi^GX\ar[d] \\
X_{hG} \ar[r] & X^{hG} \ar[r] & X^{tG},
} \end{align}
where $\Phi^G X$, which we define as $(\tilde EG \wedge X)^G$, is the spectrum of geometric fixed points, and $X^{tG}=( \tilde EG \wedge F(EG_+,X))^G$ is the Tate spectrum of $X$.

The map $X_{hG} \to X^{hG}$ is often called the ``norm map", and can be thought of as the spectrum level version of the algebraic norm map considered above. Thus the Tate spectrum is the cofiber of the norm map.

Associated to the homotopy orbit, homotopy fixed point, and Tate spectrum are spectral sequences
\begin{equation}\label{ss:GreenleesMay}
\begin{aligned}
	H_s(G,\pi_t(X)) \Rightarrow \pi_{t-s} X_{hG} \\
	H^s(G,\pi_t(X)) \Rightarrow \pi_{t-s} X^{hG} \\
	\hat H^s(G,\pi_t(X)) \Rightarrow \pi_{t-s} X^{tG} .
\end{aligned}
\end{equation}

If $R$ is a ring spectrum, then so are $R^{hG}$ and $R^{tG}$; moreover, the homotopy fixed point and Tate spectral sequences are spectral sequences of differential algebras and there is a natural map between them that is compatible with the differential algebra structure. 

From the definition it is clear that $X^{tG}$ is contractible if and only if the norm map is an equivalence $X_{hG} \simeq X^{hG}$ between homotopy orbits and homotopy fixed points. We will show that this is the case for $X=KU$ with the $C_2$-action by complex conjugation.

\section{The homotopy fixed point spectral sequence for $KO\simeq KU^{hC_2}$}\label{sec:HFPSS}

In the equivalence $KO = KU^{hC_2}$, $C_2$ acts on $KU$ by complex conjugation. This action has been made ``genuinely" equivariant by Atiyah \cite{atiyah1966k}, who constructed a $C_2$-equivariant version of $KU $ often denoted by $K\R$ indexed by representation spheres. In particular, the fixed points $K\R^{C_2}$ are equivalent to the homotopy fixed points $KU^{hC_2}=KO$. In view of this, we will abandon the notation $K\R$ and refer to this ``genuine" $C_2$-spectrum by the name $KU$ as well. 

The $E_2$-term of the homotopy fixed point spectral sequence for $KU^{hC_2}$ is given by $H^*(C_2;KU_*)$, with the generator $c$ acting on $KU_*=\Z[u^{\pm 1}]$ via $c(u)= -u$. The standard projective resolution 
\[ \cdots \xrightarrow{1-c} \Z[C_2]\xrightarrow{1+c} \Z[C_2]\xrightarrow{1-c} \Z \]
of $\Z$ immediately gives that the $E_2$-term of the HFPSS can be described as 
\[
E_2^{s,t} = \Z[\eta,t^{\pm 1}] /(2\eta),
\]
with $|\eta| = (1,1)$ and $|t| = (4,0)$, where $t = u^2$ is the square of the Bott class and $\eta$ is the class that ends up representing the Hopf element $S^1\to KO$.

It has become somewhat standard to compute the differentials in the HFPSS as follows. Via Bott periodicity we know that $KO$ is an 8 periodic ring spectrum, hence, an analysis of the spectral sequence forces there to be a differential $d_3(t)=\eta^3$, and the algebra structure determines the rest of the differentials.
Even though this is a famous differential, we include here a proof from scratch that we learned from Mike Hopkins and Mike Hill and which only uses the fact that $KO\simeq KU^{hC_2} \simeq KU^{C_2}$, as well as the equivariant form of Bott periodicity. The importance of the argument lies in its potential to be generalized to higher chromatic analogues of $KO$, such as $TMF$ or $EO_n$'s, where less is known a priori.

The HFPSS is the Bousfield-Kan spectral sequence for the cosimplicial spectrum $ F((EC_2)_+, KU)$, i.e. 
\[\xymatrix{ KU  \ar@<0.5ex>[r]^-c\ar@<-0.5ex>[r]_-1 & F((C_2)_+, KU)  \ar@<0ex>[r] \ar@<1ex>[r]\ar@<-1ex>[r] & F((C_2)^2_+, KU) \ar@<0.5ex>[r]\ar@<-0.5ex>[r] \ar@<1.5ex>[r]\ar@<-1.5ex>[r] &\cdots }\]
Its associated normalized complex (or $E_1$ page of the spectral sequence) is 
\[\xymatrix{KU \ar[r]^{1-c} & KU \ar[r]^{1+c} & KU \ar[r]^{1-c} &\cdots   }\]

Let $\rho$ denote the regular complex representation of $C_2$; then $\rho=1\oplus \sigma$, with $\sigma$ the sign representation. The one-point compactification $S^{\rho}$ of $\rho$ is $C_2$-equivalent to $\C P^1$. As such, it gives a $C_2$-map $S^{\rho} \hookrightarrow BU(1) \to BU \to KU$ which we will denote by $v_1$. This $v_1$ defines an element in $\pi_{\rho} KU $ which corresponds to a $KU$-module map
\[S^{\rho }\wedge KU \to KU \]
which is the $C_2$-equivalence known as Bott periodicity~\cite[XIV.3.2]{may1996equivariant}.

An equivalent way to construct $v_1$ is by observing that the composition $S^1 \xrightarrow{\eta} KO \to KU$ is null and $C_2$-equivariant (where, of course, $S^1 $ and $KO$ have trivial actions). Thus it extends over the interior $D^2$ of $S^1$ in two ways; if $\tau:D^2\to KU$ is one of them, then $c \circ \tau:D^2\to KU$ is the other. By construction, $\tau$ and $c\circ \tau$ agree on the boundary $S^1$ of $D^2$, hence can be glued to give a map $S^{\rho} \to KU$ which is precisely $v_1$, as $cu=-u$.

Our goal is to show that in the HFPSS for $KU$, $d_3 (v_1^2)=\eta^3$, and our method is to look at the map of HFPSSs induced by $v_1^2:S^{2\rho} \to KU$. This will be fruitful because the HFPSS for $S^{2\rho}$ is determined by the cell decomposition of $(S^{2\rho})^{hC_2}$.

For any integer $n$, we have that the homotopy orbits $(S^{-n\sigma})_{hC_2}$ are equivalent to the Thom spectrum $(\R P^\infty)^{-n\xi}$, where $\xi$ is the tautological line bundle of $\R P^\infty$. This Thom spectrum is briefly denoted $\R P^\infty_{-n}$, is called a stunted projective space, and has been extensively studied. In particular, its cell structure is well known. The schematic of the cell structure for $n=2$ is as follows
\begin{align}\label{eq:celldiagram1}
\xymatrix{ 
\underset{-2}{\bullet} \ar@{-}@/^1pc/[rr]& \underset{-1}\bullet  \ar@{-}@/_1pc/[rr] \ar@{-}[r] & \underset{0}\bullet  &\underset{1}\bullet \ar@{-}[r]  & \underset{2}\bullet  \ar@{-}@/^1pc/[rr]  & \ar@{-}[r] \underset{3}\bullet  \ar@{-}@/_1pc/[rr] & \underset{4}\bullet &\ar@{-}[r]  \underset{5}\bullet & \underset{6}\bullet \cdots
}
\end{align}
where the bullets denote cells in dimension given by the corresponding number, straight short lines between two bullets indicate that the cells are attached via multiplication by $2$, and the curved lines indicate attaching maps $\eta$. The pattern is four-periodic.

Let $DX$ denote the Spanier-Whitehead dual of $X$. We have a chain of equivalences
\[(S^{2\rho})^{hC_2}  \simeq  (D S^{-2\rho})^{hC_2} \simeq D((S^{-2\rho})_{hC_2}) \simeq D(\Sigma^{-2} (S^{-2\sigma})_{hC_2} ) \simeq \Sigma^2 D(\R P^\infty_{-2}). \]
Hence the cell diagram of $(S^{2\rho})^{hC_2} $ is
\begin{align}
\xymatrix{
\cdots \underset{-4}\bullet \ar@{-}[r]  & \underset{-3}\bullet  \ar@{-}@/_1pc/[rr]  & \underset{-2}\bullet \ar@{-}[r] \ar@{-}@/^1pc/[rr] & \underset{-1}\bullet &\underset{0}\bullet \ar@{-}[r] &\underset{1}\bullet  \ar@{-}@/_1pc/[rr]  &\underset{2}\bullet \ar@{-}[r] \ar@{-}@/^1pc/[rr] &\underset{3}\bullet &\underset{4}\bullet 
}
\end{align}
i.e. it is the diagram in~\eqref{eq:celldiagram1} flipped and shifted by $2$.

This determines the $E_1$-page of the HFPSS for $S^{2\rho}$ and all differentials in HFPSS for $KU$. Let us describe how that works.~\Cref{fig:HFPSS1} shows the relevant parts of the $E_1$-terms for the HFPSS for $S^{2 \rho}$ and $KU$, the first following from the cell decomposition above. In particular, the $d_1$ differential is non-zero where shown in which case it is simply given by multiplication by 2. The map $v_1^2$ relates the spectral sequences, and in particular, as noted, the generator labelled $e_4$ maps to the generator labeled $t:=v_1^2$.

It is a simple matter now to determine the relevant parts of the $E_2$-page. Note that there is a possible $d_2$ differential on the $S^{2 \rho}$ spectral sequence, whilst sparsity precludes such a differential on $KU^{hC_2}$. In fact, the attaching map $\eta$ in the cell decomposition above determines the $d_2$ differential $d_2(e_4) = \eta h_1^2$ on the $S^{2 \rho}$ spectral sequence, as shown in~\Cref{fig:HFPSS2}.

Under the map of spectral sequences induced by $v_1^2$, the classes $e_4$ and $h_1^2$ map to the classes $v_1^2=t$ and $\eta^2$, respectively. Moreover, multiplication by $\eta$ remains multiplication by $\eta$ under $v_1^2$, but it is a map of cohomological degree $1$ in the HFPSS for $KU$. In particular, we get the differential $d_3(v_1^2) = \eta^3$ as expected; multiplicativity now gives that
\[ d_3(v_1^{2+4k} \eta^m) = v_1^{4k}\eta^m d_3(v_1^2) = v_1^{4k} \eta^{m+3}, \] 
i.e. the usual differential pattern depicted in~\Cref{fig:e3ss}.

\begin{figure}[tbh]
\centering
\includegraphics[scale=0.93]{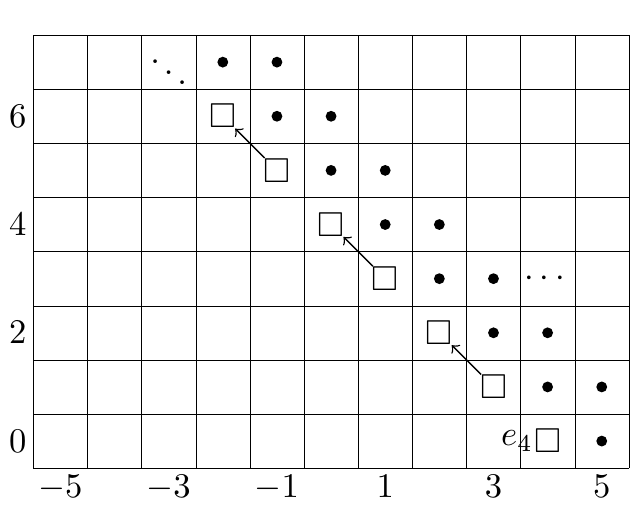}
\includegraphics[scale=0.93]{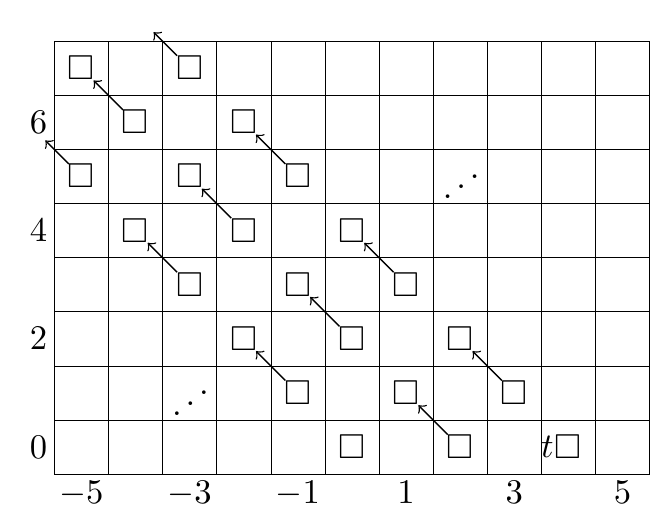}
\caption{The $E_1$-pages for the HFPSS for $S^{2 \rho}$ and $KU$. By equivariant Bott periodicity, $v_1^2$ induces an isomorphism in degree $(4,0)$.}\label{fig:HFPSS1}
\end{figure}

\begin{figure}[tbh]
\centering
\includegraphics[scale=0.93]{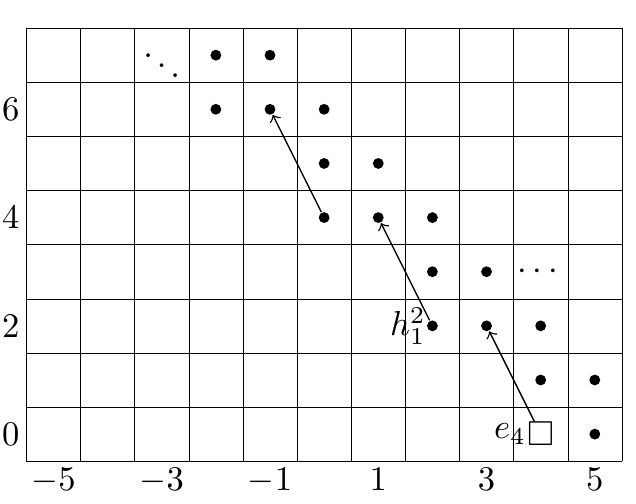}
\includegraphics[scale=0.93]{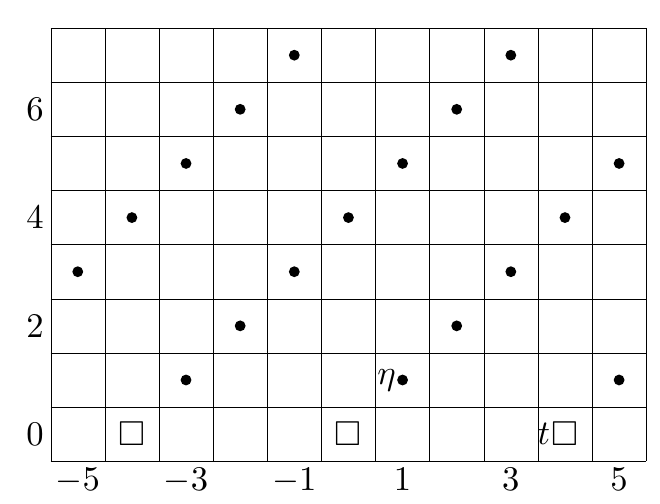}
\caption{The $E_2$-pages for the HFPSS for $S^{2 \rho}$ and $KU$. Note that there is a $d_2$-differential on the $E_2$-page for $S^{2 \rho}$ but there is no room for a $d_2$ differential for $KU$.}\label{fig:HFPSS2}
\end{figure}

\begin{figure}[tbh]
\centering
\includegraphics{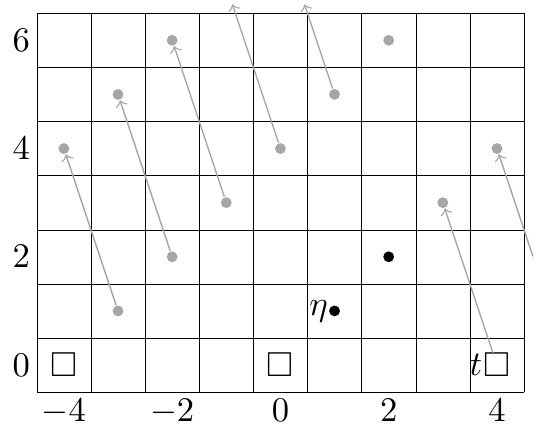}
\caption{The $E_3$-page for the HFPSS for $KU$.}\label{fig:e3ss}
\end{figure}

\section{The Tate spectrum and the Euler class}\label{sec:tate}

In addition to the spectral sequences \eqref{ss:GreenleesMay}, there is a conceptual way to compute the homotopy of a Tate spectrum from that of the homotopy fixed points and it involves a (co)homology class called the Euler class, described below. The calculations in this section bear some similarity to those of Fajstrup~\cite{tatefajstrup} who calculates that the generalized Tate spectrum associated to $K\R$ is equivariantly contractible. The key to both calculations is the fact that the generator $\eta \in \pi_1 KO$ is nilpotent.

We now define the Euler class. Think of the sign representation $\sigma$ as a $C_2$-bundle over a point. Then the Euler class is a characteristic class \[e_\sigma \in H^1_{C_2}(\ast) = H^1(\R P^\infty) \simeq \Z/2.\]   Alternatively, the Euler class $e_{\sigma} $ is the element of 
\[\pi_{-\sigma}^{C_2} KU = [S^0, \Sigma^{\sigma} KU]^{C_2}\cong [S^1,\Sigma^{\rho}KU]^{C_2} \cong [S^1,KU]^{C_2} \cong \pi_1 KO=\Z/2,\]
which is defined as the $C_2$-fixed points of composition of the inclusion of fixed points $S^0 \to S^{\sigma}$ with the $\sigma$-suspension of the unit map $S^0\to KU$. Note that in the chain of equivalences above, we have used the equivariant form of Bott periodicity. We claim that $e_{\sigma}$ is the non-trivial element of $\pi_{-\sigma}^{C_2} KU$; indeed, the inclusion $S^0\to S^{\sigma}$ is (unstably) the identity on $C_2$-fixed points, and the fixed points of the unit map $S^0\to KU $ contain the unit map for $KU^{C_2}\simeq KO$, which is non-trivial.  

Thus we conclude that $e_{\sigma} = \eta \in \pi_1 KO$.
\begin{prop}\label{thm:tatevanish}
The Tate spectrum $KU^{tC_2}$ as well as the geometric fixed point spectrum $\Phi^{C_2} KU$ are both contractible.
\end{prop}
\begin{proof}
By~\cite[XVI.6.8]{may1996equivariant}, $\Phi^{C_2} KU \simeq KU^{C_2}[e_{\sigma}^{-1}] $ and $KU^{tC_2}\simeq KU^{hC_2}[e_{\sigma}^{-1}]$. But $e_{\sigma}=\eta$ is nilpotent in $KO\simeq KU^{C_2} \simeq KU^{hC_2}$, whence the claim.
\end{proof}

We can also show that the Tate spectrum $KU^{tC_2} $ vanishes by a direct computation. This will have the advantage of making the homotopy orbit spectral sequence easy to calculate. From the above discussion, the $E_2$-term of the Tate spectral sequence for $KU$ can be described as
\[
\hat E_2^{s,t} = E_2^{s,t}[\eta^{-1}]= \Z/2[\eta^{\pm 1},t^{\pm 1}].
\]
We proved above that $\eta^3=0$ in the homotopy of $KO$. Since the map between the homotopy fixed point and Tate spectral sequence is compatible with the differential algebra structure, we also have that $\eta^3=0$ in the Tate spectral sequence. 
Thus we must have $d_3(t) = \eta^3$ and the differentials then follow the pattern shown in~\Cref{fig:TateSS} because of multiplicativity. In particular the $E_4$-page is 0, hence the Tate spectrum is contractible as the Tate spectral sequence is conditionally convergent.

\begin{figure}[hbt]
\centering
\includegraphics{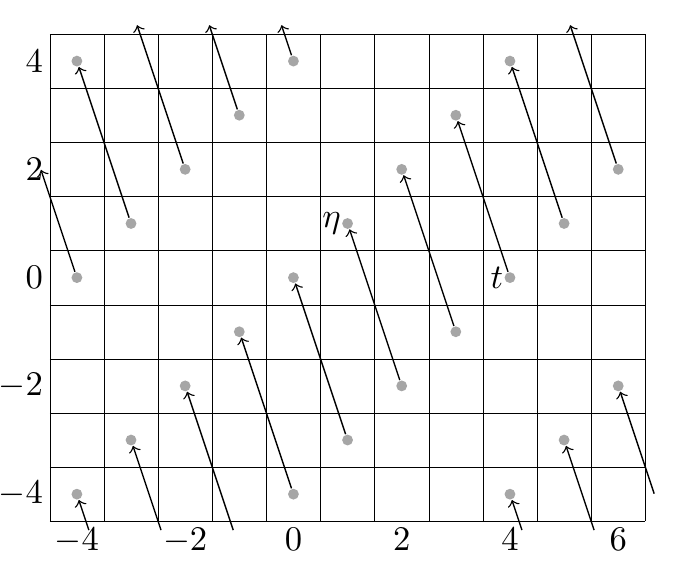}
\caption{Tate spectral sequence for $KU$}\label{fig:TateSS}
\end{figure}

The vanishing of the Tate spectrum allows us to prove the following two corollaries which we used in~\Cref{sec:andersonduality}. 
\begin{prop}
Let $M$ be a $C_2$-equivariant $KU$ module. Then $M^{tG}$ and $\Phi^G M$ are contractible.
\end{prop}
\begin{proof}
The spectra $M^{tG}$ and $\Phi^G M$ are modules over $KU^{tG}$ and $\Phi^G KU$, respectively, and these are contractible by~\Cref{thm:tatevanish}.
\end{proof}

\begin{cor}\label{cor:FixedHtpyFixed}
Let $M$ be a $C_2$-equivariant $KU$ module. Then the fixed points $M^{C_2}$ and the homotopy fixed points $M^{hC_2}$ are equivalent.
\end{cor}
\begin{proof}
Follows immediately from the diagram \eqref{eq:NormCofibration}, in which the two rightmost spectra become contractible by the above proposition.
\end{proof}

\section{The homotopy orbit spectrum for $KO$}

To calculate the homotopy orbit spectral sequence we can use the Tate cohomology groups as given in~\Cref{eq:tateDef} as well as the calculations shown in~\Cref{fig:TateSS}. We just need to calculate the co-invariants; for the trivial action $H_0(C_2;\Z) \simeq \Z$, whilst for the sign-representation $H_0(C_2,\Z_{sgn}) = \Z/2$.

We draw the homotopy orbit spectral sequence with cohomological grading, i.e. set $E_2^{s,t} = H_{-s}(C_2,\pi_tKU)$. The differentials can be inferred from those in the Tate spectral sequence. The resulting spectral sequence is shown in~\Cref{fig:HOSS}. Note that there has to be an additive extension as shown with the dashed line, since $\pi_0 KO \simeq \Z$ is torsion free. 
\begin{figure}[tbh]
\centering
\includegraphics{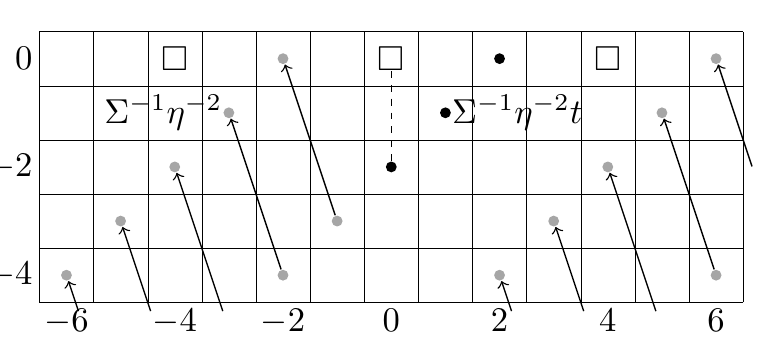}
\caption{The homotopy orbit spectral sequence for $KU$.}\label{fig:HOSS}
\end{figure}

\section{The Anderson dual of $KO$}\label{sec:KOanderson}

From~\Cref{eq:AD,thm:tatevanish}, we get the following sequence of equivalences
\[
I_\Z KO \simeq F(KU^{hC_2},I_{\Z}) \simeq F(KU_{hC_2},I_{\Z}) \simeq F(KU,I_{\Z})^{hC_2}.
\]
Note that even though $I_{\Z} KU \simeq KU$, this does not imply that $(I_\Z KU)^{hC_2} \simeq KO$, as the action may have changed. However, the action on the homotopy groups of $I_\Z KU$ is unchanged as both the trivial representation and the sign representation of $C_2$ are self-dual. 
Hence the $E_2$-terms of the HFPSS associated to $KU$ and $I_\Z KU$ are isomorphic, and a twist in the action would alter the differential pattern. Since the $E_2$-term of the HFPSS's is 4-periodic, and the HFPSS of $I_{\Z} KU$ is a module over the one for $KU$, we have that $I_\Z KO$ is either $KO$ or $\Sigma^4 KO$. We will show that it is the latter by determining the pattern of differentials.

\begin{theorem}\label{theorem:andersonKO}
The Anderson dual of $KO$ is $\Sigma^4 KO$.
\end{theorem}
\begin{proof}
The argument here will follow the proof of Theorem 13.1 of~\cite{stojanoska2012duality}. Our goal is to show that the spectral sequence computing $F(KU_{hC_2},I_{\Z})$ is the linear dual of the homotopy orbit spectral sequence for $KO = KU_{hC_2}$. 

As in the proof of~\cite[Thm.13.1]{stojanoska2012duality} (which in turn is modeled after Deligne's ``Lemma of two filtrations"~\cite{Deligne}), we can construct a commuting square of spectral sequences
\[
\xymatrix{
\Ext^v_\Z(H_h(C_2,\pi_t KU),\Z) \ar@{=>}[r]^*+[Fo]{B}
 \ar@{=>}[d]_*+[Fo]{A} & \Ext^v_\Z(\pi_{t+h}(KU)_{hC_2},\Z) \ar@{=>}[d]^*+[Fo]{D} \\
H^{h+v}(C_2,\pi_{-t}I_\Z KU) \ar@{=>}[r]_*+[Fo]{C} & \pi_{-t-h-v} I_\Z(KU_{hC_2}).
}
\]
Here $B$ is dual to the homotopy orbit spectral sequence, $A$ is obtained via the Grothendieck composite functor spectral sequence associated to the functors $\Hom$ and coinvariants, and $D$ is the Anderson duality spectral sequence given by~\Cref{eq:andersonSS}. We want to show that the spectral sequences $B$ and $C$ are isomorphic. As explained in~\cite{stojanoska2012duality}, the differentials
in $B$ are compatible with the filtration giving $C$ if and only if $A$ collapses, which holds in our case. Consequently, $B$ and $C$ are isomorphic. Thus we apply $\text{Ext}_{\Z}(-,\Z)$ to homotopy orbit spectral sequence to obtain the homotopy fixed point spectral sequence for computing the Anderson dual of $KO$, and we can read the differentials straight from the homotopy orbit spectral sequence. 

A diagram of the spectral sequence can be seen in~\Cref{fig:HFPSS}. This gives an isomorphism on homotopy $\pi_*(I_\Z KO) \simeq \pi_*(\Sigma^4 KO)$. Invoking~\cref{prop:moduleeq} now produces an equivalence $\Sigma^4KO \simeq I_\Z KO$.
\end{proof}
\begin{figure}[h!]
\centering
\includegraphics{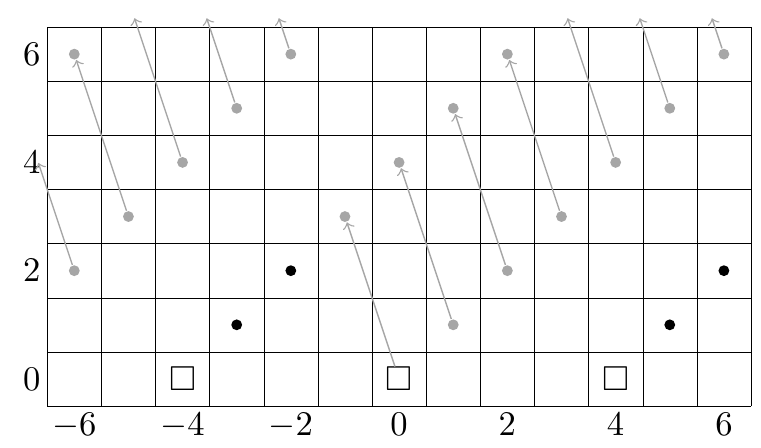}
\caption{The homotopy fixed point spectral sequence for $I_\Z KU$}\label{fig:HFPSS}
\end{figure}

We are finally in a position to prove~\Cref{thm:Andersondual}. 
\begin{theorem}\label{thm:Andersondual}
The Anderson dual $F(KU, I_\Z)$ is $C_2$-equivariantly equivalent to $\Sigma^4KU$.
\end{theorem}

\begin{proof}
By~\Cref{prop:moduleeq}, we can construct a non-equivariant equivalence $\Sigma^4 KU \to I_\Z KU$ which is nevertheless a $KU$-module map. We claim that this map has an equivariant refinement, or equivalently, that the spectrum map $d: S^4 \to I_\Z KU$ generating the dualizing class has an equivariant refinement.
To see that, consider the $C_2$-cofiber sequence
\[ C_{2+} \to S^0 \to S^{\sigma},\]
which after mapping to $I_\Z KU$ on $\pi_4^{C_2}$ gives an exact sequence 
\[ [S^4, I_\Z KU ]_{C_2} \to [S^4, I_\Z KU ] \to [S^{\sigma+3},I_\Z KU ]_{C_2}. \]
It suffices to show that $d \in [S^4, I_\Z KU]$ maps to zero in $[S^{\sigma+3},I_\Z KU ]_{C_2}$. But using equivariant Bott periodicity and~\Cref{cor:FixedHtpyFixed}, the latter is equivalent to \[[S^2, I_\Z KU]_{C_2} \cong \pi_2 (I_\Z KU)^{C_2}\cong \pi_2 (I_\Z KU)^{hC_2}, \] which by~\Cref{theorem:andersonKO} is zero.
\end{proof}

\section{Applications to the $K(1)$-local Picard group}\label{sec:picard}

Fix a prime $p$; in this section we use the usual notation for the chromatic players. So, $K(n) $ and $E_n$ denote the Morava $K$ and $E$-theory spectra at height $n$; $\G_n$ is the (big) Morava stabilizer group, $L_n$ denotes localization with respect to $E_n$\footnote{or equivalently with respect to the Johnson-Wilson spectrum $E(n)$}, and $M_n$ is the monochromatic layer functor, i.e. the homotopy fiber of the map $L_n\to L_{n-1}$.

As currently defined, even if $X$ is a $K(n)$-local spectrum, $I_{\Q/\Z}X$ need not be. The relevant details to properly extend Brown-Comenetz duality to the $K(n)$-local category have been worked out by Gross and Hopkins~\cite{hopkins1994rigid}, and this form of duality also goes by the name Gross-Hopkins duality. 
For $K(n)$-local $X$, the Gross-Hopkins dual of $X$ is defined as $$I_nX = F(M_nX,I_{\Q/\Z});$$
This spectrum is $K(n)$-local; indeed by~\cite[Proposition 2.2]{stojanoska2012duality} it is in fact isomorphic to $L_{K(n)}F(L_nX,I_{\Q/\Z})$. 

Since $I_\Q X$ becomes contractible after $K(n)$-localisation (for $n \ge 1$) we get that, if $X$ is $E_n$-local, the fiber sequence in~\Cref{eq:fibSeq} gives
\begin{equation}\label{eq:andersonBC}
I_nX = L_{K(n)}\Sigma I_{\Z} X.
\end{equation}

Since $KO$ is $E_1$-local we immediately obtain the following statement.
\begin{cor}\label{cor:K1localBCdual}
The $K(1)$-local Brown-Comenetz dual of $KO$ is given by 
\[
I_1 KO \simeq \Sigma^5 KO.
\]	
\end{cor}

The other main type of duality in the $K(n)$-local category is $K(n)$-local Spanier-Whitehead duality, defined by $D_nX = F(X,L_{K(n)}S^0)$, and now we relate it to Gross-Hopkins duality. To do so we need the Picard group of the $K(n)$-local homotopy category, introduced by Hopkins~\cite{hopkins1994constructions}. We call a $K(n)$-local spectrum $X$ invertible if there exists a spectrum $X^{-1}$ such that $L_{K(n)}(X \wedge X^{-1}) \simeq L_{K(n)}S^0$; the Picard group $\text{Pic}_n$ is the group of isomorphism classes of invertible objects. There is a map $\epsilon: \text{Pic}_n \to H^1(\G_n,(E_n)_0^\times)$, and if $p$ is large compared to $n$ then this map is injective~\cite[Prop.7.5]{hopkins1994constructions}. 

 \addtocounter{footnote}{1}

Hopkins, Mahowald and Sadofsky further show that a $K(n)$-local $X$ being invertible is equivalent to $(E_n)_*X$ being a free $(E_n)_*$-module of rank 1.\footnote{Note that we follow the usual convention in that $(E_n)_*X:=\pi_*L_{K(n)}(E_n \wedge X)$.} Noting that there is an equivalence $M_n X \simeq M_n L_{K(n)}X$, let $I_n$ denote $I_n L_{K(n)}S^0$. By work of Gross and Hopkins~\cite{hopkins1994rigid}, $(E_n)_*I_n$ is a free $(E_n)_*$-module of rank 1 and hence $I_n $ defines an element of $ \text{Pic}_n$. In particular $I_n$ is dualisable, so there is a $K(n)$-local equivalence
\begin{equation}\label{eq:grossSW}
I_nX \simeq F(X,I_n) \simeq  D_nX \wedge I_n .
\end{equation}

In fact more can be said by incorporating the action of the Morava stabilizer group $\G_n$. We recall briefly that there is a reduced determinant map $\G_n \to \Z_p^\times$, and if $M$ is a $(E_n)_*$-module with a compatible action of $\G_n$, then we write $M \sdet$ for the module with $\G_n$ action twisted by the determinant map. Then Hopkins and Gross in fact prove the stronger statement (see also~\cite{strickland2000gross}) that
\[
(E_n)_*I_n \simeq \Sigma^{n^2-n}(E_n)_* \sdet.
\]
There is a `twisted sphere' spectrum $S^0 \sdet$ (see~\cite[Remark 2.7]{goerss2012hopkins}) such that $(E_n)_* S \sdet \simeq (E_n)_* \sdet$ and thus whenever $\epsilon:\text{Pic}_n \to H^1(\G_n,(E_n)_0^\times)$ is injective there is an equivalence
\begin{equation}\label{eq:hopgross}
I_n \simeq \Sigma^{n^2-n} S^0\sdet.
\end{equation}

We will show that at $n=1,p=2$ this does \emph{not} hold, and hence that the map $\epsilon$ has non-trivial kernel traditionally denoted $\kappa_1$ and called the `exotic' $K(1)$-local Picard group. We note that at $n=1,  S^0\langle\text{det} \rangle \simeq S^2$, so trivial kernel would imply that $I_1 \simeq S^2$.

From now on, we set $p=2$ and omit $K(1)$-localization from the notation with the understanding that everything is in the $K(1)$-local category. Note that the Morava $E$-theory spectrum $E_1$ is now the $2$-completion\footnote{equivalently, $K(1)$-localization} of $KU$, the Morava stabilizer group $\G_1$ is $\Z_2^\times$ with its elements corresponding to $2$-adic Adams operations. The subgroup $C_2=\{ \pm 1\} \in \Z_2^\times $ is a maximal finite subgroup whose non-trivial element acts on $KU$ by $\psi^{-1}$ which is complex conjugation. 

Let $l$ be a topological generator of $\Z_2^{\times}/\{\pm 1\}$; it is well known (eg.~\cite{bou79}) that the $K(1)$ local sphere sits in a fiber sequence
\[ S \to KO \xrightarrow{\psi^l-1} KO.\]
Taking the Anderson dual of this sequence gives
\[ \Sigma^4KO \xrightarrow{I_\Z\psi^l -1} \Sigma^4KO \to \Sigma^{-1}I_1. \]
Let $P$ denote $\Sigma^{-2}I_1$; hence $P$ is an element of $\Pic_1$ such that $P\neq S$ if and only if $I_1 \neq S^2$.

\begin{lemma}\label{lem:dualAdams}
The Anderson dual of the Adams operation $\psi^l$ is \[I_\Z\psi^l=\Sigma^4(l^{-2}\psi^{1/l}) = \Sigma^{-4}(l^2 \psi^{1/l}) .\]
\end{lemma}

\begin{proof}
The idea of the proof is that the self-maps of $KO$ are detected by their effect on homotopy. To be more precise, we recall that there is an equivalence
\[ KO[[\Z_2^\times/\{\pm 1\}]] \to F(KO,KO), \]
where the left hand side is defined as $\holim_H (KO \wedge H)$, where $H$ runs over the finite quotients of the profinite group $\Z_2^\times/\{\pm 1\}$ (a reference is, for example, \cite[Prop.2.6]{GHMR}). In particular, this describes the homotopy classes of maps $KO \to KO$ as the completed group ring $\Z_2[[\Z_2^\times/\{\pm1\}]]$, i.e. power series on the $2$-adic Adams operations of $KO$, and these can be distinguished by what they do on the non-torsion elements of $KO_*$.

Let $t \in KO_{4n}$ be an arbitrary non-torsion element of $KO_*$. Then the Adams operation $\psi^l$ acts on $t$ as
\[
\psi^l(t) = l^{2n}t.
\]
Note that therefore, $\Sigma^{8m} \psi^l:KO\to KO$ is $l^{-4m} \psi^l$.

Since the linear dual of multiplication by $m$ is again multiplication by $m$, we see that
\[
I_{\Z} \psi^l:I_\Z KO \to I_\Z KO
\]
is given on the non-torsion homotopy by
\[
I_\Z \psi^l(t^\vee) = l^{2n} t^\vee.
\]
Identifying $\Sigma^{-4}I_\Z KO$ with $KO$, we get that 
\[ \Sigma^{-4}I_\Z \psi^l:KO \to KO \]
takes an element $t\in KO_{4n} \cong (I_\Z KO)_{4n+4}$ and maps it to 
$ l^{-2n-2} t = l^{-2} \psi^{1/l}(t) $.
\end{proof}
Note that, in particular, this lemma provides resolutions of $P$ as 
\begin{equation}\label{eq:P-resolution}
\begin{aligned}
P \to \Sigma^4KO \xrightarrow{\Sigma^4(l^{-2}\psi^{1/l})-1} \Sigma^4 KO, \text{ or} \\
P \to\Sigma^{-4}KO \xrightarrow{\Sigma^{-4}(l^{2}\psi^{1/l})-1} \Sigma^{-4} KO,
\end{aligned}
\end{equation}
which are a $K(1)$-local analogue (us to a suspension shift) of the resolution of the $K(2)$-local Brown-Comenetz spectrum of Goerss-Henn~\cite{goerss2012brown}. We smash the second resolution with $\Sigma^{4} KO$ to deduce the following result.

\begin{prop}\label{prop:PKO}
The smash product $P\wedge KO$ is $\Sigma^4 KO$. Consequently, $P$ defines a non-trivial element of $\kappa_1$.
\end{prop}

\begin{proof}
We consider the fiber sequence of $KO$-modules
\[ \Sigma^{4}P \wedge KO \to KO\wedge KO \xrightarrow{\alpha}  KO\wedge KO,\]
from which we will determine $\pi_*\Sigma^{4}P\wedge KO$, where for brevity we have denoted by $\alpha$ the map $(l^{2}\psi^{1/l}-1) \wedge 1$. The homotopy groups of $KO\wedge KO$ are (see \cite[Prop.1]{hopkins1998k} or \cite[Prop.2.4]{GHMR})
\[  \pi_* (KO\wedge KO) \cong KO_0KO \otimes_{KO_0} KO_* \cong \Mapc(\Z_2^\times/\{ \pm 1\},\Z_2) \otimes_{\Z_2} KO_*, \]
and the map $\alpha$ sends $f \otimes x \in \Mapc(\Z_2^\times/\{ \pm 1\},\Z_2) \otimes_{\Z_2} KO_* $ to $g\otimes x$, where $g \in \Mapc(\Z_2^\times/\{ \pm 1\},\Z_2)$ is the function determined by
\[ g(k) = l^2 f(k/l) - f(k). \]
Consequently, $f\otimes x$ is in the kernel of $\pi_*\alpha$ if and only if for every $k\in \Z_2^\times/\{\pm 1\}$, we have $f(k)=l^2 f(k/l)$. In particular, the kernel of $\pi_0 \alpha$ is $\Z_2$, where $a\in \Z_2$ corresponds to the continuous function $f_a$ determined by $f_a(1)=a$ and the relation $f_a(k)=l^2 f_a(k/l)$; thus the kernel of $\pi_* \alpha $ is $\Ker \pi_0\alpha \otimes KO_*$.

Next, we claim that $\pi_0 \alpha$, and therefore $\pi_*\alpha$, is surjective. Indeed, let $g$ be a function $\Z_2^\times/\{\pm 1 \} \to \Z_2$; then the function defined by
\[ f(k) = \sum_{n=1}^\infty \frac{1}{l^{2n}} g(l^nk )  \]
is such that $\pi_0(\alpha)(f)=g$.

We conclude that $\pi_*\Sigma^{4}P\wedge KO \cong \Ker\pi_0\alpha \otimes KO_* \cong KO_*$ as a $KO_*$-module. Appealing to~\cref{prop:moduleeq} now gives us that $\Sigma^{4}P\wedge KO \simeq KO$.
\end{proof}

\begin{prop}\label{prop:ordertwo}
The element $P $ of $\kappa_1$ has order two, i.e. $P\wedge P \simeq S$.
\end{prop}
\begin{proof}
The proof is similar to the one for~\cref{prop:PKO}; we smash the resolution of $P$ in~\cref{eq:P-resolution} with $P$ and use the description of $P\wedge KO$ from the proof of~\cref{prop:PKO}. Denote again by $\alpha$ the map $(l^2\psi^{1/l}-1)\wedge 1:KO\wedge KO$, and let $\beta$ denote the map $1\wedge (l^2\psi^{1/l}-1):KO\wedge KO$. Consider the commutative diagram
\[\xymatrix{
\Sigma^{-4}KO\wedge\Sigma^{-4}KO \ar[r]^{\Sigma^{-8}{\beta}} \ar[d]_{\Sigma^{-8}\alpha} & \Sigma^{-4}KO\wedge\Sigma^{-4}KO\ar[d]^{\Sigma^{-8}\alpha}\\
\Sigma^{-4}KO\wedge\Sigma^{-4}KO \ar[r]^{\Sigma^{-8}{\beta}} & \Sigma^{-4}KO\wedge\Sigma^{-4}KO;
} \]
we know that the vertical fibers are $P\wedge \Sigma^{-4} KO\simeq KO$, and the fiber of the induced map between them is $P\wedge P$.

Let us again identify the homotopy groups of $KO\wedge KO$ with $\Mapc(\Z_2^\times/\{\pm 1\},\Z_2)\otimes KO_*$; then (as in~\cite{hopkins1998k}) $\beta$ takes a function $f$ to the function $\beta f$ defined by
\[\beta f(k) = l^2 \psi^{1/l}(f(kl)) - f(k). \]
Now suppose that $f$ is in the kernel of $\alpha$, i.e. it is in the homotopy of the fiber of $\alpha$; then for every $k$, $f(kl)=l^2f(k)$, whence
\[\beta f(k) =l^4 \psi^{1/l}(f(k)) -f(k).  \]
This implies that the restriction of $\Sigma^{-8}\beta$ to the fibers of $\Sigma^{-8}\alpha$ is 
\[\Sigma^{-8}(l^4\psi^{1/l}- 1):  P\wedge\Sigma^{-4} KO \simeq KO \to P\wedge\Sigma^{-4} KO \simeq KO. \]
But we have that, as in the proof of~\Cref{lem:dualAdams},
\[\Sigma^{-8}(l^4\psi^{1/l} ) = \frac{1}{l^4} l^4 \psi^{1/l} = \psi^{1/l}, \]
so our fiber sequence becomes
\[ P\wedge P \to KO \xrightarrow{\psi^{1/l}-1 } KO, \]
showing that $P\wedge P\simeq S$, as $1/l$ is also a topological generator of $\Z_2^\times/\{\pm 1 \}$.
\end{proof}

The effect of~\Cref{prop:PKO,prop:ordertwo} is that we have produced a subgroup $\Z/2$ of $\kappa_1$ generated by $P$. It is known by the work of Hopkins-Mahowald-Sadofsky~\cite{hopkins1994constructions} that in fact this is all of $\kappa_1$; for completeness we include here what is essentially their construction of a surjective map $\kappa_1\to \Z/2$, thus recovering the following result.
 	
\begin{prop}
	The exotic Picard group $\kappa_1$ is $\Z/2 $ generated by $P$.
\end{prop}
\begin{proof}
Let $Z$ be an arbitrary element of $\kappa_1$; the key point is that $KU_*Z $ and $ KU_*$ are isomorphic as $KU_*$-modules, and the isomorphism respects the $\Z_2^\times$-action. (Recall that we are working $K(1)$-locally, so $KU$ really means the $2$-completion of $KU$.) Therefore the $E_2$-term for the $K(1)$-local Adams-Novikov spectral sequence for $Z$ coincides with that for the sphere, although the differentials may be different. By~\cite[Rem.3.4]{goerss2012hopkins} (see also the work of~\cite{hoveysadofsky,shimomura}) there is a group homomorphism \[\tau:\kappa_1 \to H^3(\Z_2^\times,(KU)_2) \simeq \Z/2,\] defined in the following way. Let $\iota_Z \in H^0(\Z_2^\times,(KU)_0Z) \simeq \Z_2$ be the identity element. The first (and indeed the only) possible differential in the $K(1)$-local ANSS for $Z$ is a $d_3$. Given a choice of a $\Z_2^\times$-equivariant isomorphism $f:KU_* \xrightarrow{\simeq} KU_*Z$, we have a diagram
\[
\xymatrix{
H^0(\Z_2^\times,(KU)_0) \ar@{-->}[r]^{\phi}  \ar[d]_{f_*}^{\simeq} &H^3(\Z_2^\times,(KU)_2) \ar[d]_{\simeq}^{f_*} \\
H^0(\Z_2^\times,(KU)_0Z) \ar[r]_{d_3} & H^{3}(\Z_2^\times,(KU)_2Z).
}
\]
and we define $\tau(Z) = \phi(\iota_Z) = f_*^{-1} d_3 f_*(\iota_Z)$. This  does not depend on the choice of $f$ and defines the claimed group homomorphism $\kappa_1 \to H^3(\Z_2^\times,(KU)_2) \simeq \Z/2$. 

Note that $Z \simeq L_{K(1)}S^0$ if and only if $\iota_Z$ survives the spectral sequence. The $E_2$-term of the spectral sequence for $Z$ is a free module of rank one over the $E_2$-term of that for the sphere, generated by the class $\iota_Z$, and this fully determines the $d_3$ differential for $Z$. Standard calculations (for example~\cite{hoveysadofsky}) show that $E_4=E_\infty$ and therefore the only differential possible is a $d_3$. This implies that $P$ maps to the non-trivial element of $\Z/2$ under the map $\tau$, and so $\tau$ is a surjection. It is also injective for if $d_3(\iota_Z) = 0$, then $\iota_Z$ is a permanent cycle and the resulting map extends to an equivalence $Z \simeq L_{K(1)}S^0$. 
\end{proof}

As an additional application, we can use~\Cref{prop:PKO,prop:ordertwo} to compute the $K(1)$-local Spanier-Whitehead dual of $KO$, thus recovering the result of~\cite[Lemma 8.16]{hahn2007iwasawa}.
\begin{cor}
	The $K(1)$-local Spanier-Whitehead dual of $KO$ is given by $D_1(KO) \simeq \Sigma^{-1} KO$. 
\end{cor}
\begin{proof}
	We have the series of equivalences
	\begin{align*}
	D_1KO &\simeq F(KO,S) \simeq F(KO,P) \wedge P \\
	&\simeq F(KO,\Sigma^{-1}I_\Z) \wedge P 	\simeq \Sigma^{-1} \Sigma^4 KO \wedge P \\
	&\simeq \Sigma^{-1} KO.
\end{align*}
\end{proof}

\bibliographystyle{amsalpha}
	\bibliography{biblio}

\end{document}